\documentclass[11pt]{amsart}
\usepackage[margin=1in]{geometry}  
\usepackage{longtable}
\usepackage[table]{xcolor}
\usepackage{amsmath, amsthm, latexsym, amssymb, amsfonts}
\usepackage{alltt}
\usepackage[pdftex]{graphicx} 
\usepackage[all]{xy}

 \theoremstyle{plain}
 \newtheorem{tm}{Theorem}[section]
 \newtheorem{lm}[tm]{Lemma}
\newtheorem{coro}[tm]{Corollary}

 \theoremstyle{definition}
 \newtheorem{defi}[tm]{Definition}
 \newtheorem{rema}[tm]{Remark}
 
  \newtheorem{ex}[tm]{Example}
  \newtheorem{pdr}[tm]{Procedure}
  
 \newtheorem*{th*}{Theorem}

\newcommand{\cl}[1]{\mathcal{#1}}

\newcommand{\Q}{\mathbb Q}
\newcommand{\Z}{\mathbb Z}
\newcommand{\N}{\mathbb N}
\newcommand{\R}{\mathbb R}
\newcommand{\F}{\mathbb F}
\newcommand{\Pp}{\mathbb P}

\newcommand{\cA}{\cl A}

\newcommand{\K}{{\mathbb K}}

\newcommand{\la}{\langle}
\newcommand{\ra}{\rangle}

\newcommand{\m}{\mathfra\K{m}}

\newcommand{\Cl}{\operatorname{Cl}}

\newcommand{\dis}{\displaystyle}
\def\aa{{\bf \alpha}}
\def\bb{\beta}

\def\t{{\bf t}}
\def\q{{\bf q}}
\def\y{{\bf y}}
\def\z{{\bf z}}
\def\uu{{\mathbf{u}}}
\def\vv{{\bf v}}
\def\x{{\bf x}}
\def\q{{\bf q}}
\def\m{{\bf m}}

\def\cc{{\bf c}}

 \usepackage{graphics}
 \usepackage[english]{babel}

\usepackage[latin1]{inputenc}

\usepackage{times}
\usepackage[T1]{fontenc}
 \usepackage{graphics}
\usepackage{algorithm}
\usepackage{algpseudocode}

\begin{document}


\title{Vanishing ideals of parameterized subgroups in a toric variety}
 \thanks{The author is supported by T\"{U}B\.{I}TAK-2211}

\author[Esma BARAN \"{O}ZKAN]{Esma BARAN \"{O}ZKAN}
\address[Esma BARAN \"{O}ZKAN]{Department of Mathematics, \c{C}ank{\i}r{\i} Karatekin University, \c{C}ank{\i}r{\i}, TURKEY}
\email{esmabaran@karatekin.edu.tr}


\begin{abstract}Let $\K$ be a finite field and $X$ be a complete simplicial toric variety over $\K$. We give an algorithm relying on elimination theory for finding generators of the vanishing ideal of a subgroup $Y_Q$ parameterized by a matrix $Q$ which can be used to study algebraic geometric codes arising from $Y_Q$. We give  a method  to compute the lattice $L$ whose ideal $I_L$ is exactly $I(Y_Q)$ under a mild condition. As applications, we give precise descriptions for the lattices corresponding to some special subgroups. We also prove a Nullstellensatz type theorem valid over finite fields, and share \verb|Macaulay2| codes for our algorithms.
\end{abstract}

\maketitle

\section{Introduction}Let $X$ be a  complete simplical toric variety determined by the fan $\Sigma$ over a finite field $\K=\F_q$ whose class group $\Cl(X)$ is torsion-free. Let $\vv_1,\dots,\vv_r\in \Z^n$ denote the  primitive generators of the one dimensional cones $\rho_1,\dots,\rho_r$ in $\Sigma$, respectively.  Recall from  \cite[Theorem 4.1.3]{CLSch} that when $\vv_1,\dots,\vv_r$ span   $\R^n$, the sequence
\begin{equation}\label{td1}\dis \xymatrix{ \mathfrak{P}: 0  \ar[r] & \Z^n \ar[r]^{\phi} & \Z^r \ar[r]^{{\bb}} & \Z^d \ar[r]& 0},\end{equation}
is exact, where $\phi$ has rows $\vv_1,\dots,\vv_r$ and $d=r-n$. To simplify notation, we write $\textbf{t}^\textbf{a}=t_1^{a_1}\cdots t_s^{a_s}$ for any $\textbf{a}=(a_{1},\dots,a_{s})\in\Z^s$. Dualizing the short exact sequence above yields the   exact sequence
\begin{equation}\label{td2}\dis \xymatrix{ \mathfrak{P}^*: 1  \ar[r] & G \ar[r]^{i} & (\K^*)^r \ar[r]^{\pi} & (\K^*)^n \ar[r]& 1},\end{equation}
where $\pi$ sends $(t_1,\dots, t_r)\in (\K^*)^r $ to $ (\mathbf{t}^{\uu_1}, \dots , \mathbf{t}^{\uu_n})$ such that  
$\uu_1,\dots, \uu_n$ are  columns of $\phi.$ It follows that the torus $T_X$ is isomorphic to quotient of $(\K^*)^r$ by the group $G=\ker(\pi)$. Then every element of $T_X$ can be defined as $[p_1:\cdots:p_r]:=G\cdot (p_1,\dots,p_r)$  for some $(p_1,\dots,p_r)\in (\K^*)^r$.
$X$ has  total coordinate ring  $S=\K[x_1,\dots,x_r]$ graded by the group $\Cl(X)\cong\Z^d$. The degree of a monomial $\x^\textbf{a}=x_1^{a_1}\cdots x_r^{a_s}$ is
$$\deg_{\Z^d}(\x^\textbf{a}):=a_1\deg_{\Z^d}(\x_1)+\cdots a_r\deg_{\Z^d}(\x_r)=a_1\beta_1+\cdots a_r\beta_r$$
such that $\beta_j:=\beta(e_j).$ The grading on $S$ yields a direct sum decomposition $S=\bigoplus_{\aa \in \N\beta} S_{\aa}$ where $S_{\aa}$ is vector space whose basis consists of  the monomials having degree $\aa$. 

For any matrix $Q=[\q_1 \q_2\cdots \q_r]\in M_{s\times r}(\Z)$, the subset
$Y_Q=\{[{\t}^{\q_{1}}:\cdots:{\t}^{\q_{r}}]\:|\:\t\in (\K^{*})^s\}$ of $T_X$ is called the toric set parameterized by $Q$. The subgroups in $X$  are exactly these  parametric sets by \cite[Theorem 3.2 and Corollary 3.7]{Sahin}. Given a subset $Y\subset X$, the vanishing ideal $I(Y)$ in $S$ is the graded ideal generated by homogenous polynomials in $S$ which vanish at all the elements in $Y$. Vanishing ideals play an important role in Algebraic Geometry and Coding Theory.  Computing  the generators of the vanishing ideal $I(Y_Q)$ makes it easy to determine its algebraic and geometric properties and to compute the basic parameters of a linear code obtained by evaluating homogeneous polynomials at the points of $Y_Q$, see  \cite{Algebraicmethodsforparameterizedcodesandinvariantsofvanishingoverfinitefields,SPR,BaranSahin,sasop}. 

The aim of this paper is to reveal that some of the results scattered in the literature about the vanishing ideal of $Y_Q$ can be generalised to the more general toric case by carefully adopting and modifying the arguments used. The problem of giving an algorithm computing minimal generators for $I(Y_Q)$ was solved by Villarreal, Simis and Renteria in \cite[Theorem 2.1]{Algebraicmethodsforparameterizedcodesandinvariantsofvanishingoverfinitefields}, when $X$ is a projective space. When $Y_Q$ is the torus $T_X$ of a weighted projective space $X$, Dias and Neves proved in \cite{Codesoveraweightedtorus} a generalized version of \cite[Theorem 2.1]{Algebraicmethodsforparameterizedcodesandinvariantsofvanishingoverfinitefields}. Our first generalization is Theorem \ref{t:elim}, which gives an expression for $I(Y_Q)$, leading to a method (see  Algorithm \ref{a:elim}) via Elimination Theory for computing a generating set. We provide a \verb|Macaulay2| code  for  Algorithm \ref{a:elim} in Procedure \ref{pdr:elim}. 

In \cite[Theorem 2.5]{Algebraicmethodsforparameterizedcodesandinvariantsofvanishingoverfinitefields}, the authors prove that
$I(Y_Q)=I_{L}$ for the lattice $L=L_Q+(q-1)L_{\bb}$ under a mild condition on $Q$. Applying this result they also proved in \cite[Proposition 4.3]{Algebraicmethodsforparameterizedcodesandinvariantsofvanishingoverfinitefields} that $Y_Q$ is nothing but the zero set in the torus $T_X$ of the lattice ideal $I_{L_Q}$, and concluded in \cite[Corollary 4.4]{Algebraicmethodsforparameterizedcodesandinvariantsofvanishingoverfinitefields} that $I(V_X(I_{L}))=I_{L}$. We generalize all these in Theorem \ref{t:hold} and Theorem \ref{tm:nullstellensatz}.

Lopez, Villarreal and Zarate computed in \cite{LVZ} the generators of the vanishing ideal of $Y_Q$ by determining  the lattice $L$ for toric sets parameterized by diagonal matrices in projective space. In \cite{Sahin}, \c{S}ahin not only showed that the vanishing ideal $I(Y_Q)$ is a lattice ideal for more general toric varieties but also generalized this result about diagonal matrices to more general toric varieties. We conclude the paper by giving an alternative proof in Theorem \ref{t:degenerate}.


\section{Vanishing Ideal via Elimination Theory} \label{S:elim} 
In this section, we give a method yielding an algorithm for computing the generators of the vanishing ideal of $Y_Q$. The following basic theorems will be used to obtain our first main result giving this method.

\begin{lm}\label{l:null} \cite[Lemma 2.1]{CombinatorialNullstellensatz} Let $\K$ be a field and $f$ be a polynomial in $\K[x_1,\dots,x_s]$ such that  $deg_{x_{i}}f\leq k_i$ for all $i$. Let  $\K_i\subseteq \K$ be   a finite set with $\lvert \K_i\rvert\geq k_{i}+1$ for $1\leq i\leq s$. If $f$ vanishes on $\K_1\times \K_2\times \cdots \times \K_s$, then $f=0$.
\end{lm}

\begin{tm}\label{t:elimination}\cite[Theorem 2, p.87]{IdealsVarietiesandAlgorithms} Let $I\subset \K[x_1,\dots,x_k]$ be an ideal and let
	$G$ be a Gr\"obner basis of I with respect to lex order where $x_1> x_2> \cdots > x_k$ . Then,
	for every $0\leq l \leq k$ the set
	$$G_l = G \cap \K[x_{l+1},\dots,x_k]$$
	is a Gr\"obner basis of the $l$-th elimination ideal $I_l=I \cap \K[x_{l+1},\dots,x_k]$.
\end{tm}

When $X=\Pp^n$, $S$ is graded by $\cA=\Z$ via $\deg(x_i)=1$ for all $i$. In this case,  there is a method to compute the vanishing ideal $I(Y_Q)$, see \cite[Theorem 2.1]{Algebraicmethodsforparameterizedcodesandinvariantsofvanishingoverfinitefields}. Recently, Tochimani and Villarreal \cite{TV} generalized this result to subsets of $X=\Pp^n$ that are parameterized by ratios of Laurent polynomials, whose proof does not generalize to all toric varieties. 
When $X=\Pp(w_1,\dots,w_r)$ is the weighted projective space, $S$ is also $\Z$-graded via $\deg(x_i)=w_i\in \N$ using the degree map $\bb=[w_1 \cdots w_r]$. In \cite{Codesoveraweightedtorus}, Dias and Neves  extended the same result to weighted projective spaces if $Q$ is the identity matrix, that is, $Y_Q=T_X$.  Our main contribution is to make necessary arrangements needed for a general toric variety to adapt the corresponding proof in \cite[Theorem 2.1]{Algebraicmethodsforparameterizedcodesandinvariantsofvanishingoverfinitefields}. Before we generalize these results to arbitrary toric varieties, let us recall that $\m=\m^+-\m^-$, where $\m^+,\m^-\in \N^r$, and $\x^{\m}$ denotes the monomial $x_1^{m_1}\cdots x_r^{m_r}$ for any $\m=(m_1,\dots,m_r)\in \Z^r$.

\begin{tm} \label{t:elim2}Let $R=\K[x_1,\dots,x_r,y_1,\dots,y_s, z_1,\dots, z_d,w]$ be a polynomial ring which is  an extension of $S$. Then  $I(Y_Q)=J\cap S $, where
	$$J=\la\{x_i{\y}^{{{\q}_i}^{-}}\textbf{z}^{{\beta_i}^-}-{\y}^{{{\q}_i}^{+}}\textbf{z}^{{\beta_i}^+}\}^{r}_{i=1}\cup\{y_i^{q-1}-1\}^{
		s}_{i=1},w{\y}^{{{\q}_1}^{-}}{\z}^{{\beta_1}^-}\cdots {\y}^{{{\q}_r}^{-}}{\z}^{{\beta_r}^-}-1\ra.$$\end{tm}
\begin{proof} We follow the steps in the proof of \cite[Theorem 2.1]{Algebraicmethodsforparameterizedcodesandinvariantsofvanishingoverfinitefields} making some necessary modifications. 
	$I(Y_Q)$ is generated by homogeneous polynomials, since it  is a homogeneous ideal. To show the inclusion $I(Y_Q)\subseteq J\cap S$, pick any generator $f=\sum\limits_{i=1}^{k} c_i \x^{\m_i}$ of degree $\alpha=\sum_{j=1}^{r}\beta_jm_{ij}$. We use binomial theorem to write any monomial $\x^{\m_i}$ as
	$$\begin{array}{lcl}
		\x^{\m_{i}} &=&x_1^{m_{i1}}\cdots x_r^{m_{ir}}=\left( x_1-\dfrac{{\y}^{{{\q}_1}^{+}}\textbf{z}^{{\beta_1}^+}} {{\y}^{{{\q}_1}^{-}}\textbf{z}^{{\beta_1}^-}}+\dfrac{{\y}^{{{\q}_1}^{+}}\textbf{z}^{{\beta_1}^+}}{{\y}^{{{\q}_1}^{-}}\textbf{z}^{{\beta_1}^+}}\right) ^{m_{i1}}\cdots\left( x_r-\dfrac{{\y}^{{{\q}_r}^{+}}\textbf{z}^{{\beta_1}^+}} {{\y}^{{{\q}_r}^{-}}\textbf{z}^{{\beta_1}^-}}+\dfrac{{\y}^{{{\q}_r}^{+}}\textbf{z}^{{\beta_r}^+}}{{\y}^{{{\q}_r}^{-}}\textbf{z}^{{\beta_r}^-}}\right)^{m_{ir}} \; \;\\
		
		&=&\sum\limits_{j=1}^{r} g_{ij}\dfrac{\ x_j{\y}^{{{\q}_j}^{-}}\textbf{z}^{{\beta_j}^-}-{\y}^{{{\q}_j}^{+}}\textbf{z}^{{\beta_j}^+}}{\left({\y}^{{{\q}_1}^{-}}\textbf{z}^{{\beta_1}^-}\right)^{m_{i1}}\cdots  \left({\y}^{{{\q}_r}^{-}}\textbf{z}^{{\beta_r}^-}\right)^{m_{ir}}}+ \textbf{z}^{\alpha}\left( \dfrac{{\y}^{{{\q}_1}^{+}}} {{\y}^{{{\q}_1}^{-}}}\right) ^{m_{i1}}\cdots\left( \dfrac{{\y}^{{{\q}_r}^{+}}} {{\y}^{{{\q}_r}^{-}}}\right) ^{m_{ir}}\\
	\end{array}$$
	where $g_{i1},...,g_{ir} \in R$ for all $i$. Substituting all the $\x^{\m_i}$ in $f$, we get
	$$ 
	f=\sum\limits_{j=1}^{r} g_{j}\dfrac{\ x_j{\y}^{{{\q}_j}^{-}}\textbf{z}^{{\beta_j}^-}-{\y}^{{{\q}_j}^{+}}\textbf{z}^{{\beta_j}^+}}{\left({\y}^{{{\q}_1}^{-}}\textbf{z}^{{\beta_1}^-}\right)^{m'_{1}}\cdots  \left({\y}^{{{\q}_r}^{-}}\textbf{z}^{{\beta_r}^-}\right)^{m'_{r}}} +\textbf{z}^{\alpha}f({\y}^{q_1},\dots,{\y}^{q_r}),\\
	$$
	where $g_j=\sum\limits_{i=1}^{k}\,\,c_j g_{i,j}$ and $m'_j=\sum_{i=1}^k m_{ij}$. Then set  $m=\sum_{j=1}^{k}\sum_{i=1}^{r}m_{ij}$. We multiply $f$ by $h^{m}$ to clear all the denominators which yields that
	$$fh^{m}=\sum\limits_{j=1}^{r}G_j\left(x_j{\y}^{{{\q}_j}^{-}}\textbf{z}^{{\beta_j}^-}-{\y}^{{{\q}_j}^{+}}\textbf{z}^{{\beta_j}^+}\right)+\textbf{z}^{{\alpha}'}{\left({{\y}^{{{\q}_1}^{-}}}\cdots {{\y}^{{{\q}_r}^{-}}}\right)}^{m}f({\y}^{q_1},\dots,{\y}^{q_r}), $$ 
	where $h={\y}^{{{\q}_1}^{-}}{\z}^{{\beta_1}^-}\cdots {\y}^{{{\q}_r}^{-}}{\z}^{{\beta_r}^-}$ and $${\alpha}'=\sum_{j=1}^{r} \left[{\beta_j}^{-} (m-m_{ij}) + {\beta_j}^{+}m_{ij} \right] \in \N^n, \quad {G}_j=g_j\prod\limits_{j=1}^{r}\left({\y}^{{{\q}_j}^{-}}\textbf{z}^{{\beta_j}^-}\right)^{m-m'_{j}}\in\;\; R.$$ We can apply division algorithm and divide $F={\left({{\y}^{{{\q}_1}^{-}}}\cdots {{\y}^{{{\q}_r}^{-}}}\right)}^{m}f({\y}^{q_1},\dots,{\y}^{q_r})$ by $\{{{y_i}^{q-1}}-1\}_{i=1}^{s}$ which leads to 
	\begin{equation}\label{e:esma}
		\resizebox{.9\hsize}{!}
		{$fh^{m}=\sum\limits_{j=1}^{r}G_i\left(x_j{\y}^{{{\q}_j}^{-}}\textbf{z}^{{\beta_j}^-}-{\y}^{{{\q}_j}^{+}}\textbf{z}^{{\beta_j}^+}\right)+\textbf{z}^{{\alpha}'}\left({\sum\limits_{i=1}^{s}H_i({y_i}^{q-1}-1)}+E(y_1,\dots,y_s) \right).$}
	\end{equation}
	
	Lemma \ref {l:null} forces that $E$ is the zero polynomial, as $E(\t)=0$ for all $\t=(t_1,\dots,t_s)\in {(\K^{*})}^s$.  
	
	Multiplying now the equation (\ref{e:esma}) by ${w}^m$, we have 
	$$f(hw)^{m}=\sum\limits_{j=1}^{r}{w}^mG_j\left(x_j{\y}^{{{\q}_j}^{-}}\textbf{z}^{{\beta_j}^-}-{\y}^{{{\q}_j}^{+}}\textbf{z}^{{\beta_j}^+}\right)+{w}^m\textbf{z}^{{\alpha}'}{\sum_{i=1}^{s}H_i({y_i}^{q-1}-1)}.$$
	As $f(hw)^{m}=f(hw-1+1)^m=f(H(hw-1)+1)=fH(hw-1)+f$ for some $H\in R$ by binomial theorem, it follows that
	$$f=\sum\limits_{j=1}^{r}{w}^{m}G_i\left(x_j{\y}^{{{\q}_j}^{-}}\textbf{z}^{{\beta_j}^-}-{\y}^{{{\q}_j}^{+}}\textbf{z}^{{\beta_j}^+}\right)+{w}^{m}\textbf{z}^{{\alpha}'}{\sum_{i=1}^{s}H_i({y_i}^{q-1}-1)}-fH(hw-1)\\
	$$
	which establishes the inclusion $I(Y_Q)\subseteq J\cap S$.
	
	As $J$ is generated   by binomials, so is $J\cap S$ by Theorem \ref{t:elimination}. Take $f= \textbf{x}^\textbf{a}- \textbf{x}^\textbf{b} \in J\cap S$ and write
	\begin{equation}\label{my_second_eqn} f=\sum\limits_{j=1}^{r}G_j\left(x_j{\y}^{{{\q}_j}^{-}}\textbf{z}^{{\beta_j}^-}-{\y}^{{{\q}_j}^{+}}\textbf{z}^{{\beta_j}^+}\right)+{\sum\limits_{i=1}^{s}H_i({y_i}^{q-1}-1)}+H(hw-1),
	\end{equation}
	for some polynomials $G_1,\dots,G_r, H_1, \dots, H_s, H$ in $R$. As the last equality is valid also in the ring $R[{z_1}^{-1},\dots,{z_d}^{-1}]$ setting $y_i=1, x_i=\textbf{z}^{\beta_i}, w=1/\textbf{z}^{{\beta_1}^-}\cdots\textbf{z}^{{\beta_r}^-}$ gives
	$$\textbf{z}^{a_1\beta_1+\cdots+a_r\beta_r}-\textbf{z}^{b_1\beta_1+\cdots+b_r\beta_r}=0.$$
	Thus, $f=\textbf{x}^\textbf{a}- \textbf{x}^\textbf{b}$ is homogeneous. By setting $x_i=\t^{\q_i}, y_j=t_j,z_k=1$ for all $i,j,k$ and $w=1/\t^{\q_1^{-}} \cdots \t^{\q_r^{-}}$ in (\ref{my_second_eqn}), we obtain $f(\t^{\q_1},\dots,\t^{\q_r})=0$ and so $f\in I(Y_Q)$. Consequently, $J\cap S\subseteq I(Y_Q)$.
\end{proof}

\begin{rema} Given a toric variety $X$, we have the matrix $\phi$. There are many alternatives for the matrix $\bb$ such that $\mathfrak{P}$ is exact. However, by \cite[Theorem 8.6 and Corollary 8.8]{CombinatorialCommutativeAlgebra}, we can choose $\beta_j\in \N^d$, where $\N$ is the set of non-negative integers. 
\end{rema}

\begin{ex} The matrix $\bb$ obtained by the following procedure have negative entries. 
	\begin{verbatim}
		i1 : q=11; Phi=matrix{{1,0},{0,1},{-1,2},{0,-1}}; r=numRows Phi;
		i2 : n=numColumns Phi; (D,P,K) = smithNormalForm Phi; Beta=P^{n..r-1};
		o2 = | -1  2 -1  0 |
		|  0  1  0  1 |
	\end{verbatim}
	Note that multiplying the first row by $-1$ and then adding twice the second row we get $\bb=\begin{bmatrix}
		1 & 0 & 1& 2\\
		0 & 1 & 0& 1  
	\end{bmatrix}$ to be a better alternative!
\end{ex}

In practice, different matrices $Q$ may give the same subgroup $Y_Q$ of $T_X$. As we work over the finite field $\F_q$, we note that restricting to matrices $Q$ with non-negative entries strictly smaller than $q-1$, does not harm the generality, since $t^{q-1}\equiv 1 \: \mbox{mod} \: q$, for any $t\in \F_q^*$. This reduction together with the assumption that $\beta_j\in \N^d$ will decrease the complexity of the algorithm for finding a generating set for $I(Y_Q)$ which will be based on the following result whose proof is skipped. 

\begin{tm} \label{t:elim}Let $R=\K[x_1,\dots,x_r,y_1,\dots,y_s, z_1,\dots, z_d]$ be a polynomial ring which is  an extension of $S$. Then  $I(Y_Q)=J\cap S $, where
	$$J=\la\{x_i-{\y}^{{{\q}_i}}\textbf{z}^{{\beta_i}}\}^{r}_{i=1}\cup\{y_i^{q-1}-1\}^{
		s}_{i=1}\ra.$$\end{tm}

Theorem \ref{t:elim} gives rise to Algorithm \ref{a:elim} for computing the binomial generators of the ideal $I(Y_Q)$ in the ring $\F_q[x_1,\dots,x_r]$ from the given matrices $Q$ and $\bb$ with entries from the set $\N$ of non-negative integers, for every prime power $q$. 

\begin{algorithm}
	\caption{ Computing the generators of vanishing ideal $I(Y_Q)$.}\label{a:elim}
	\begin{flushleft}
		\hspace*{\algorithmicindent}\textbf{Input} The matrices $Q\in M_{s\times r}(\N)$, $\beta\in M_{d\times r}(\N)$ and a prime power $q$.\\
		\hspace*{\algorithmicindent}\textbf{Output} The generators of $I(Y_Q)$.
	\end{flushleft}
	\begin{algorithmic}[1]
		\State Write the ideal $J$ of $R$  using Theorem \ref{t:elim}.
		\State Find the Gr\"obner basis $G$ of $J$ wrt. lex order  $z_1>\cdots>z_d>y_1>\cdots>y_s>x_1>\cdots>x_r$.
		\State Find $G\cap S$ so that $I(Y_Q)=\la G\cap S \ra$.
	\end{algorithmic}
\end{algorithm}

Using the function \verb|toBinomial| creating a binomial from a list of integers (see \cite{ComputationsinAlgebraicGeometrywithMacaulay2}), we write a \verb|Macaulay2| code which implements this algorithm.

\begin{pdr}\label{pdr:elim}Given a particular input $q$, $Q,\beta$, the following procedure find the generators of $I(Y_Q)$.
	\begin{verbatim}	
		i2 : toBinomial = (b,R) -> (top := 1_R; bottom := 1_R;
		scan(#b, i -> if b_i > 0 then top = top * R_i^(b_i)
		else if b_i < 0 then bottom = bottom * R_i^(-b_i)); top - bottom);		
		i3 : r=numColumns Q;s=numRows Q; d=numRows Beta; F=ZZ/q;		
		i4 : C=(id_(ZZ^r)| -transpose Q | -transpose Beta);	
		i5 : R=F[x_1..x_r,y_1..y_s,z_1..z_d]; 	
		i6 : J = ideal apply(entries C, b -> toBinomial(b,R))+ideal apply (s,i->R_(r+i) ^(q-1)-1)
		i7 : IYQ=eliminate (J,for i from r to r+s+d-1 list R_i)
	\end{verbatim}
\end{pdr}

\begin{ex}
	\label{ex:gb} 
	Let $X=\cl H_{2}$ be the Hirzebruch surface. The fan for $X$  has primitive ray generators $\vv_1=(1,0)$, $\vv_2=(0,1)$, $\vv_3=(-1,2)$, $\vv_4=(0,-1)$ and its coordinate ring $S=\K[x_1,x_2,x_3,x_4]$. Then for $X$, the matrices  $\phi$ and $\beta$ in \ref{td1} are given by \\
	$$\phi=\begin{bmatrix}
		1 & 0\\
		0 & 1\\
		-1 & 2\\
		0 &  -1  
	\end{bmatrix},\:\beta=\begin{bmatrix}
		1 & 0 & 1& 2\\
		0 & 1 & 0& 1  
	\end{bmatrix}.$$
	This shows that $S=\K[x_1,x_2,x_3,x_4]$ is graded by letting  $$\deg_{\Z^2}(x_1)=\deg_{\Z^2}(x_3)=(1,0),\quad \deg_{\Z^2}(x_2)=(0,1), \quad \deg_{\Z^2}(x_4)=(2,1).$$
	Since the map $\pi$ in \ref{td2} is defined by $\pi:\t\mapsto (t_1t_3^{-1},t_2t_3^{2}t_4^{-1})$ for $X$, $$G=\ker(\pi)=\{(t_1,t_2,t_1,t_1^{2}t_2)\;|\;t_1,t_2\in\K^*\}\cong (\K^*)^2.$$ Thus the torus of $X=X_\Sigma$ has quotient representation  $ T_X \cong {(\K^*)^2}\cong(\K^*)^{4} /G$. Consider the set parameterized by $Q=[1\;2\;3\;4]$, that is,
	$Y_Q=\{[t:t^2:t^3:t^4]~~~|~~ t\in \K^*\}$.
	In order to compute the generators of $I(Y_Q)$ in \verb|Macaulay2| it suffices to supply $q$ together with $Q$ and $\beta$:
	\begin{verbatim}
		i1 : q=11; Beta=matrix {{1,0,1,2},{0,1,0,1}}; Q=matrix {{1,2,3,4}};
	\end{verbatim}	
	\noindent and we obtain $I(Y_Q)=\la x_1^{2}{x_2}-{x_4},x_1^{5}-x_3^{5}\ra$ by using  Procedure \ref{pdr:elim}.
\end{ex}

\section{Conceptual Descriptions of the Lattice of a Vanishing Ideal}
\label{S:ConceptualLattice}
A subgroup $L\subseteq \Z^r$ is called a lattice. For any lattice $L$, the lattice ideal corresponding to $L$, denoted $I_L$, is the  ideal generated by binomials $\x^\textbf{a}-\x^\textbf{b}$ for all $\textbf{a},\textbf{b}\in \N^r$  such that $\textbf{a}-\textbf{b}\in L$.

The vanishing ideal $I(Y_Q)$ is a lattice ideal, see \cite{Sahin}. Then the lattice associated to $I(Y_Q)$ is identified in the  following lemma.
\begin{lm}\cite[Lemma 3.2]{BaranSahin}\label{lm:equality} The ideal $I(Y_Q)$ is equal to $I_{L}$ where $L=\{\m\in L_\beta: Q\m\equiv 0\;\mbox{mod}\;(q-1)\}$.
\end{lm}
This description is not very useful in practice, for it requires some operations. In this section, we  give a handier description of the lattice of the ideal $I(Y_Q)$, in terms of $Q$ and $\beta$, under a condition on the lattice $\mathcal{L}=QL_\beta=\{Q\m| \m\in L_\beta\}$ using Lemma \ref{lm:equality}. Before stating it, let us remind that $\mathcal{L}:(q-1)=\{\m\in {\Z}^s|(q-1)\m \in \mathcal{L}\}$ and the colon module $\mathcal{L}:(q-1)\Z^s$ are the same. Given any matrix $Q$, $\ker_{\Z}Q$ is a lattice, denoted $L_Q$.


\begin{tm}\label{t:hold}
	Let ${L}=(L_Q \cap {L_\beta})+(q-1)L_\beta$. Then 
	$I_{L}\subseteq I(Y_Q)$. The equality holds if and only if $\mathcal{L}=\mathcal{L}:(q-1)$.
\end{tm} 
\begin{proof} We start with the proof of the inclusion $I_{L}\subseteq I(Y_Q)$. By the virtue of Lemma \ref{lm:equality}, it suffices to prove that $L\subseteq L_1=\{\m\in L_\beta: Q\m\equiv 0\;\mbox{ mod}\; (q-1)\}$, as $I(Y_Q)=I_{L_1}$. Take $\m\in L$. Since $L=(L_Q \cap{L_\beta})+(q-1)L_\beta\subseteq L_\beta$, we have $\m \in L_\beta$. On the other hand, we can write $\m=\m'+(q-1)\m''$ for some $\m'\in L_Q\cap L_\beta$ and $\m''\in L_\beta$. Since $Q\m=(q-1)Q\m''$, it follows that $\m\in L_1$, completing the proof of the inclusion.  
	
	Now, in order to show that $I(Y_Q)\subseteq I_{L}$ iff $\mathcal{L}=\mathcal{L}:(q-1)$, it is enough to prove that $L_1 \subseteq L$ iff $\mathcal{L}:(q-1)\subseteq \mathcal{L}$. Assume first that $L_1 \subseteq L$ and take $\textbf{z}\in \mathcal{L}:(q-1)$. This means that there exist $\m \in L_\beta$ such that $(q-1)\textbf{z}=Q\m$. So, $\m\in L_1 \subseteq L$ and we have $\m=\m'+(q-1)\m''$ for some $\m'\in L_Q\cap L_\beta$ and $\m''\in L_\beta$. Thus, $(q-1)\textbf{z}=Q\m=(q-1)Q\m''$, and we have $\textbf{z}=Q\m''\in \mathcal{L}$. Therefore, $\mathcal{L}:(q-1)\subseteq \mathcal{L}$.
	
	Suppose now that $\mathcal{L}:(q-1)\subseteq \mathcal{L}$ and let $\m \in L_1$. Then $\m \in L_\beta$ and $Q\m=(q-1)\textbf{z}$ for some $\textbf{z}\in \Z^s$. So, $\textbf{z}\in \mathcal{L}:(q-1)\subseteq \mathcal{L}$ yielding $\textbf{z}=Q\m'$ for some $\m'\in L_\beta$. Hence $Q(\m-(q-1)\m')=0$ and so, $\m-(q-1)\m'\in L_Q\cap L_\beta$. This implies that $\m=(\m-(q-1)\m')+(q-1)\m'\in L$. Hence, $L_1 \subseteq L$.
\end{proof}

We can check if the condition above is satisfied and in the affirmative case we can compute the generators of the lattice using the following code in \verb|Macaulay2|.

\begin{pdr}\label{pdr:hold} Checks if the condition $\mathcal{L}=\mathcal{L}:(q-1)$ holds and produces a basis for $L$. 
	\begin{verbatim} 
		i2: s=numRows Q;
		i3: LL=image (Q*Phi);    
		i4: if LL:(q-1)*(ZZ^s)==LL then print yes else print no;
		i5: ML=mingens ((q-1)*(image Phi)+intersect(ker Q,image Phi));
	\end{verbatim}
\end{pdr}

\begin{ex}\label{Ex:cond} Consider the Hirzebruch surface $X=\cl H_{2}$ over the field $\F_{2}$ and take $Q=[1\;2\;3\;4]$. The input is as follows: 
	\begin{verbatim} 
		i1 : q=2;Phi=matrix{{1,0},{0,1},{-1,2},{0,-1}};Q=matrix {{1,2,3,4}};
	\end{verbatim}
	\noindent and then Procedure \ref{pdr:hold} tells us that the condition is satisfied and gives $L=\la(-1,0,1,0),(-2,-1,0,1)\ra$. 
	
\end{ex}

\begin{defi}\label{d: homo}
	$Q$ is called homogeneous, if there is a matrix $A\in M_{d\times s}(\mathbb{\Q})$ such that $AQ=\beta$. 
\end{defi}

\begin{lm} \label{lm:hom1}  Let $Q=[\q_1 \cdots  \q_r]\in M_{s\times r}(\Z)$. $Q$ is homogeneous iff $L_Q \subseteq L_\beta$.
\end{lm}
\begin{proof} Suppose  that $Q$ is homogeneous. So, $AQ=\beta$ for some matrix  $A\in M_{d\times s}(\mathbb{\Q})$. Take an element $\textbf{m}$ of $L_Q$. Since $Q\textbf{m}=0$, we have $\beta\textbf{m}=AQ\textbf{m}=0$. Hence, $\textbf{m}\in L_\beta$ and thus $L_Q \subseteq L_\beta$. Conversely, assume that $L_Q \subseteq L_\beta$. Denote by $Q'$ the ${(s+d)\times r}$ matrix	 $[Q\;\beta]^T$. Then $L_{Q'}=L_Q\cap L_\beta=L_Q$ which implies that the rows of $\beta$ belong to the row space of $Q$. So, a row of $\bb$ can be expressed as a $\Q$-linear combination of the rows of $Q$. Let the following be the $i$-th row of $\beta$:
	$$ \begin{array}{lcl}a_{i1}[q_{11} \cdots q_{1r}]+\cdots+a_{is}[q_{s1} \cdots q_{sr}]&=&[(a_{i1}q_{11}+\cdots+a_{is}q_{s1})\quad \cdots \quad (a_{i1}q_{1r}+\cdots+a_{is}q_{sr})]\\
		&=&[([a_{i1}\cdots a_{is}]\q_{1}) \quad \cdots \quad ([a_{i1}\cdots a_{is}]\q_{r})].
	\end{array}$$
	So, for $A=(a_{ij})\in M_{d\times s}(\mathbb{\Q})$ we have $AQ=\beta$.
\end{proof}

The next corollary is a generalization of Theorem 2.5 in \cite{Algebraicmethodsforparameterizedcodesandinvariantsofvanishingoverfinitefields} studying the case $X=\Pp^n$.

\begin{coro}
	Let $Q=[\q_1 \cdots  \q_r]\in M_{s\times r}(\Z)$ be a homogeneous matrix and  ${L}=L_Q +(q-1)L_\beta$. Then 
	$I_{L}\subseteq I(Y_Q)$. The equality holds if and only if $\mathcal{L}=\mathcal{L}:(q-1)$.
\end{coro}

\begin{proof}
	Since $Q$ is homogeneous, $L_Q\subseteq L_\beta$ which shows that $L=L_Q\cap L_\beta+ (q-1)L_\beta=L_Q +(q-1)L_\beta$. Therefore, $I(Y_Q)=I_L$ from  Theorem \ref{t:hold}.
\end{proof}

Using Theorem \ref{t:hold}, we give other proofs of the following facts proven for the first time in \cite{Sahin}.
\begin{coro} \label{c:torusideal}
	$I(T_X)=I_{(q-1)L_\beta} $
\end{coro}
\begin{proof}
	$T_X$ is the toric set parameterized by the identity matrix $Q=I_r$. It is clear that $L_Q=\mbox{ker}_{\Z}I_r=\{0\}$. Notice also that $\mathcal{L}=QL_\beta=\{I_{r}\m| \m\in L_\beta\}=L_\beta$. Since $L_\beta$ is torsion free, the condition $\mathcal{L}=\mathcal{L}:(q-1)$ is satisfied. As $L=(q-1)L_\beta$, we have that $I(T_X)=I_{(q-1)L_\beta} $ by Theorem \ref{t:hold}.
\end{proof}

\begin{coro} The vanishing ideal of the point $[1:\cdots:1]$ is the toric ideal
	$I_{L_\beta}$.
\end{coro}
\begin{proof}
	Take $Q=\beta$. Then $Y_Q=\{[1:\cdots:1]\}$ and $L_Q=L_\beta$. So, 
	$\mathcal{L}=QL_\beta=\{Q\m| \m\in L_Q\}={0}$. Thus, $L=L_\beta+(q-1){L_\beta}=L_\beta$. This gives $I([1:\cdots:1])=I_{L_\beta}$ by Theorem \ref{t:hold}.
\end{proof}

Before giving another consequence of Theorem \ref{t:hold}, let us introduce more notation. For a homogeneous ideal $J$, we define its zero locus in $X$ to be the following:
$$V_X(J):=\{[P]\in X : F(P)=0, \mbox{for all homogeneous}\: F\in J\}.$$ 
It follows from  Lemma \ref{lm:hom1} and \cite[Proposition 2.3]{Sahin} that $Q$ is homogeneous if and only if the toric ideal $I_{L_Q}$ is homogeneous. Assuming $Q$ to be homogeneous, the zero locus $V_Q:=V_X(I_{L_Q})\cap T_X$ in $T_X$ of the toric ideal $I_{L_Q}$ is a subgroup, by \cite[Corollary 2.4]{Sahin} since we study over a finite field. 
The following consequence of Theorem \ref{t:hold} ensures that it coincides with the subgroup $Y_Q$ if the condition $\mathcal{L}=\mathcal{L}:(q-1)$ holds, generalizing \cite[Proposition 4.3]{Algebraicmethodsforparameterizedcodesandinvariantsofvanishingoverfinitefields} and \cite[Corollary 4.4]{Algebraicmethodsforparameterizedcodesandinvariantsofvanishingoverfinitefields}.
\begin{tm}[Finite Nullstellensatz] \label{tm:nullstellensatz} If $Q$ is homogeneous and ${L}=L_Q +(q-1)L_\beta$, then we have the following
	\begin{enumerate}
		\item $V_Q=V_X(I_{L})$,
		\item $V_Q=Y_Q$, if the condition $\mathcal{L}=\mathcal{L}:(q-1)$ holds,
		\item $I(V_X(I_{L}))=I_{L}$, if the condition $\mathcal{L}=\mathcal{L}:(q-1)$ holds.
	\end{enumerate} 
\end{tm}
\begin{proof} Notice that $\m\in L \iff \m=\m_1+\m_2$, for $\m_1\in L_Q$ and $\m_2\in (q-1)L_\beta$. It follows that  $$\m\in L \iff \m^{+}=\m_1^{+}+\m_2^{+}+\textbf{c} ~~\mbox{and}~~ \m^{-}=\m_1^{-}+\m_2^{-}+\textbf{c},~~\mbox{for some}~~ \textbf{c} \in \N^r. $$ 
	Turning these vectors into monomials, we get the following relation
	$$\x^{\m^{+}}-\x^{\m^{-}}=\x^{\cc}[\x^{\m_2^{+}}(\x^{\m_1^{+}}-\x^{\m_1^{-}})+\x^{\m_1^{-}}(\x^{\m_2^{+}}-\x^{\m_2^{-}})].$$ 
	Hence, the ideal $I_L=I_{L_Q}+I_{(q-1)L_\beta}$. By Corollary \ref{c:torusideal}, we also have $I_L=I_{L_Q}+I(T_X)$ so that $V_Q=V_X(I_L)$ establishing (1).
	
	Take a point $[P]=[\t^{\q_1}:\cdots:\t^{\q_r}]\in Y_Q$. We have $(\x^{\m^{+}}-\x^{\m^{-}})(P)=\t^{Q\m^+}-\t^{Q\m^-}=0$ whenever $\m\in L_Q$. So, we observe that $Y_Q \subseteq V_X(I_{L_Q})$. As $Y_Q$ is a subgroup of $T_X$, it follows that $Y_Q\subseteq V_Q$.
	Taking ideals of both sides yields $I(V_Q)\subseteq I(Y_Q)$. As $V_Q=V_X(I_L)$ by (1) above, we finally have $I_L \subseteq I(V_Q)\subseteq I(Y_Q)$. Under the assumptions, $I(Y_Q)=I_L$ by Theorem \ref{t:hold}, so the three ideals coincide in the last containment, verifying (3). By \cite[Lemma 2.8]{Sahin}, $\bar{Y}_Q=V_X(I(Y_Q))$ and $\bar{V}_Q=V_X(I(V_Q))$. As we work over a finite field, any set is closed with respect to Zariski topology and thus $Y_Q=V_Q$ proving (2).
\end{proof}
There are examples for which $Y_Q\neq V_Q$, when the condition $\mathcal{L}=\mathcal{L}:(q-1)$ is broken.
\begin{ex} \label{Ex:nullCI}
	Let $X=\cl H_{2}$ be the Hirzebruch surface over $\F_{11}$ and $Q=[1\;  2\;  3\;  4 ]$. Although $Q$ is not homogeneous, we can remedy this by adjoining $\bb$ and obtaining the homogeneous matrix $Q'=[Q~\bb]^T$. Notice that $Y_Q=Y_{Q'}$ and so $I(Y_{Q'})=\la x_1^2x_2-x_4, x_1^5-x_3^5 \ra$, see Example \ref{ex:gb} .
	
	By Theorem \ref{tm:nullstellensatz}, we have $V_{Q'}=V_X(I_L)$ for the lattice $L=L_{Q'}+(q-1)L_{\bb}$. Since $L_{Q'}=\la(2,1,0,-1)\ra$ and $L_{\bb}=\la(2,1,0,-1),(1,0,-1,0)\ra$ we have $L=\la(2,1,0,-1),(10,0,-10,0)\ra$ and hence the ideal $I_L=\la x_1^2x_2-x_4, x_1^{10}-x_3^{10} \ra$.  Applying \cite[Algorithm 1]{Sahin}, we obtain the matrix $A$ below whose subgroup $Y_A$ is precisely the zero locus $V_X(I_L)$ of the lattice ideal $I_L$:
	$$A ={\begin{bmatrix}
			-1 & 2 &  -1&  0\\
			~~0 & 1 & ~~0&  1\\
			~~0 & 10& ~~0&  0\\
			~~0 & 0 & ~~1&  0\\
			
	\end{bmatrix}}$$
	We now apply one of the algorithms developed in the previous sections to conclude that 
	$$I(V_{Q'})=I(V_X(I_L))=I(Y_A)=I_L.$$ Therefore, $V_{Q'}\neq Y_{Q'}$ as otherwise they would have the same vanishing ideals.  
\end{ex}

We close this section discussing another special case where $Q$ is diagonal.

\begin{defi}
	The toric set parameterized by a diagonal matrix is called a degenerate torus.
\end{defi}  

Let $\eta$ be a generator of the cyclic group $\K^*$, then for all $t_i \in \K^*$ we can write $t_i=\eta^{h_i}$ for some $0 \leq h_i \leq q-2$. The following is the generalized version of the corresponding result in \cite{LVZ} valid for $X=\Pp^n$. The first proof is given in  \cite{Sahin} and we give another here using Lemma \ref{lm:equality}.

\begin{tm}\label{t:degenerate} Let $Q=diag(q_1,\dots,q_r)\in M_{r\times r}(\mathbb{\Z})$ and $D=diag(d_1,\dots,d_r)$ where $d_i=|\eta^{q_i}|$. Then $I(Y_Q)=I_L$ for $L=D(L_{\beta D})$.
\end{tm} 
\begin{proof} By Lemma \ref{lm:equality}, it is enough to show that $L=L_1$ where $L_1=\{\m\in L_\beta: Q\m\equiv 0\mbox{ mod}\; (q-1)\}$. Let $\m$ be any element in $L$. Then $\m=D\textbf{z}$ for some $\textbf{z}\in \Z^r$ and $\m\in L_\beta$. Since  $d_i=(q-1)/\mbox{gcd}(q-1,q_i)$, $q_id_i\equiv 0\;\mbox{mod}\;(q-1)$ for all $d_i$. Therefore, $Q\m=QD\textbf{z}\equiv 0\;\mbox{mod}\; (q-1)$ and so, $\m\in L_1$.
	
	Take $\m \in L_1$. Then $Q\m\equiv 0\;\mbox{mod}\;(q-1)$, that is, for every $1\leq i\leq r$ there exist $z_i\in \Z$ such that  $q_im_i=(q-1)z_i$. Hence, $$\dfrac{q_i}{\mbox{gcd}(q-1,q_i)}m_i= \dfrac{q-1}{\mbox{gcd}(q-1,q_i)}z_i=d_iz_i.$$ Since ${q_i}/\mbox{gcd}(q-1,q_i)$ and ${q-1}/\mbox{gcd}(q-1,q_i)$ are coprime, it follows that $d_i$ divides $m_i$. Therefore, $m_i=d_iz'_i$ for some $z'_i\in\Z$ and so, $\m=D\textbf{z}'$ for $\textbf{z}'=(z'_1,\dots,z'_r)$. Hence, $\m\in L.$
\end{proof}
\section*{Acknowledgment}
The  author is supported by T\"{U}B\.{I}TAK-2211. The article is part of her PhD thesis carried out as part of the T\"{U}B\.{I}TAK Project 114F094 under supervision by the Principal Investigator Mesut \c{S}ahin. The author is very thankful to him for his guidance, efforts and many useful advises.

\end{document}